\documentclass[11pt]{article}
\usepackage{amsmath,amssymb,amsthm}
\usepackage[hidelinks]{hyperref}
\usepackage{cleveref}
\usepackage{setspace}

\usepackage{enumitem}
\newlist{steps}{enumerate}{1}
\setlist[steps, 1]{label = Step \arabic*:}

\newtheorem{theorem}{Theorem}[section]
\newtheorem{proposition}[theorem]{Proposition}
\newtheorem{lemma}[theorem]{Lemma}

\theoremstyle{definition}

\newtheorem{example}[theorem]{Example}

\theoremstyle{remark}

\crefname{equation}{}{}
\Crefname{equation}{Equation}{Equations}
\crefname{theorem}{Theorem}{Theorems}
\Crefname{theorem}{Theorem}{Theorems}
\crefname{lemma}{Lemma}{Lemmas}
\Crefname{lemma}{Lemma}{Lemmas}
\crefname{proposition}{Proposition}{Propositions}
\Crefname{proposition}{Proposition}{Propositions}
\crefname{corollary}{Corollary}{Corollaries}
\Crefname{corollary}{Corollary}{Corollaries}
\crefname{conjecture}{Conjecture}{Conjectures}
\Crefname{conjecture}{Conjecture}{Conjectures}
\crefname{section}{Section}{Sections}
\Crefname{section}{Section}{Sections}
\crefname{example}{Example}{Examples}
\Crefname{example}{Example}{Examples}
\crefname{problem}{Problem}{Problems}
\Crefname{problem}{Problem}{Problems}
\crefname{remark}{Remark}{Remarks}
\Crefname{remark}{Remark}{Remarks}
\crefname{figure}{Figure}{Figures}
\Crefname{figure}{Figure}{Figures}

\newcommand{\ZZ}{\mathbb{Z}}

\newcommand{\doi}[1]{\href{http://dx.doi.org/#1}{\texttt{doi:#1}}}
\newcommand{\arxiv}[1]{\href{http://arxiv.org/abs/#1}{\texttt{arXiv:#1}}}

\title{Covering of triples by quadruples \\ in bipartite and tripartite settings}

\author{Alexander Sidorenko\\
\small Department of Extremal Combinatorics \\[-0.8ex]
\small Alfr\'ed R\'enyi Institute of Mathematics, Hungary \\
\small\tt sidorenko.ny@gmail.com
}

\date{\today}

\begin{document}

\maketitle

\begin{abstract}

Let $A$ and $B$ be disjoint sets of sizes $a$ and $b$, respectively. 
Let $f(a,b)$ denote the minimum number of quadruples needed 
to cover all triples $T \subseteq A \cup B$ such that $|T \cap A| \geq 2$. 
We prove upper and lower bounds on $f(a,b)$ 
and use them to derive upper bounds for the $(n,4,3,4)$-lottery problem. 


\end{abstract}

\section{Introduction}

A system $\cal{R}$ of $r$-element subsets of an $n$-element underlying set $V$ 
is called an $(n,r,t,k)$-\emph{lottery system} 
if for every $k$-element subset $K \subseteq V$ 
one can find a member $R\in\cal{R}$ with $|R \cap K| \geq t$. 
The minimum size of such a system, denoted by $L(n,r,t,k)$, 
was studied in \cite{
Bate:1998,
Bertolo:2004,
Cushing:2023,
Droesbeke:1982,
Furedi:1996,
Hanani:1964,
Li:1999,Li:2000,Li:2002,Montecalvo:2015}. 
The smallest presently known lottery systems with parameters 
$n \leq 90$, $t \leq r \leq 15$, $k \leq 20$ can be found in \cite{Italian:tables}. 
We are primarily interested in the case $r=k=4$, $t=3$, 
and will write $L(n)$ instead of $L(n,4,3,4)$. 

An $(n,r,t,t)$-lottery system is also known as an $(n,r,t)$-\emph{covering} system. 
Its minimum size is $C(n,r,t)=L(n,r,t,t)$. 
A covering system of this size is called \emph{optimal}. 
It is known (see \cite{Fort:1958}) that 
$C(a,3,2) = \lceil\frac{a}{3}\lceil\frac{a-1}{2}\rceil\rceil$. 
To simplify notation, we will write $C(n)$ instead of $C(n,3,2)$. 

The \emph{Tur\'an number} $T(n,k,r)$ is the minimum size of a system 
of $r$-element subsets of an $n$-element set such that 
every $k$-element subset contains a member of the system. 
Obviously, $T(n,k,r) = L(n,r,r,k)$. 
A survey on the Tur\'an numbers can be found in \cite{Sidorenko:1995}. 
The exact values of $T(n,4,3)$ up to $n=18$ can be found in \cite{Markstrom:2022}. 
As every block of an $(n,4,3,4)$-lottery system contains $4$ triples, 
it is easy to see that $L(n,4,3,4) \geq \frac{1}{4}\, T(n,4,3)$. 
The famous Tur\'an's conjecture states that 
$\lim_{n\to\infty} n^{-3} \, T(n,4,3) = \frac{2}{27}$. 
As far as we know, the best lower bound for this limit is $0.07305$ 
(see \cite{Razborov:2014}).
The upper bounds on $L(n)$ which we present in this article imply  
$\lim_{n\to\infty} n^{-3} \, L(n) \leq \frac{1}{54}$. 
Hence, if the Tur\'an's conjecture is true, then 
$\lim_{n\to\infty} n^{-3} \, L(n) = \frac{1}{54}$. 

We derive our bounds on $L(n)$ from the following problem that 
is interesting in itself. 
Let $A,B$ be disjoint finite sets. 
We call a system ${\cal Q}$ of quadruples from $V = A \cup B$ 
an $(A,B)$-\emph{system} if every triple $T \subset V$ with $|T \cap A| \geq 2$ 
is contained in one of the blocks of ${\cal Q}$. 
Let $f(a,b)$ denote the minimum size of an $(A,B)$-system with $|A|=a$, $|B|=b$. 
In Sections~2-4, we will construct $(A,B)$-systems 
and prove upper and lower bounds on $f(a,b)$. 
In particular, we will find the exact values of $f(a,b)$ in the following cases:
\begin{enumerate}[label=\arabic*)]
\item 
$a$ is odd, \:$b \geq a-2$;
\item 
$a \equiv 2,6 \bmod 12$, \:$b \geq a$;
\item 
$a$ is even, $b \geq 2a-2$. 
\end{enumerate}

In \cref{sec:lotto} we use upper bounds on $f(a,b)$ with $b \geq a-3$ 
to get upper bounds on $L(n)$. 
In \cref{sec:concluding}, we discuss $(A,B)$-systems with blocks  of size $r>4$ 
and their application to the $(n,r,3,4)$-lottery problem. 

We use the notation $[n] := \{1,2,\ldots,n\}$.

\section{Lower bounds on $f(a,b)$}

A system of triples is called a \emph{packing} 
if every pair of elements is contained in at most one triple. 
Let $P(a)$ denote the size of the largest packing system on $a$ elements. 
Systems of that size will be called \emph{optimal}. 

\begin{proposition}[{\cite[Theorem~1]{Spencer:1968}}]\label{th:packing}
\begin{align*}
  P(a) = \begin {cases}
    \left\lfloor\frac{a}{3}\left\lfloor\frac{a-1}{2}\right\rfloor\right\rfloor
    \;\;\;{\rm if}\; a \not\equiv 5 \bmod 6;
    \\
    \left\lfloor\frac{a}{3}\left\lfloor\frac{a-1}{2}\right\rfloor\right\rfloor - 1
    \;\;\;{\rm if}\; a \equiv 5 \bmod 6.
  \end{cases}
\end{align*}
\end{proposition}

For a system ${\cal T}$ of triples on a finite set $A$, 
we define its \emph{weight} $w({\cal T})$ as the number of triples in ${\cal T}$ 
plus one-half the number of pairs in $A$ not covered by these triples. 
We denote by $C_*(a)$ the minimum weight of such a system when $|A|=a$. 

\begin{proposition}\label{th:C_*}
\begin{align*}
  C_*(a) =
  \begin{cases}
      \left\lceil\frac{a^2-a}{6}\right\rceil
       \;\;\;{\rm if}\; a \;{\rm is\;odd}; \\
      \frac{1}{2} \left\lceil\frac{2a^2-a}{6}\right\rceil
        \;\;\;{\rm if}\; a \;{\rm is\;even}.
  \end{cases}
\end{align*}
\end{proposition}

\begin{proof}[\bf{Proof}]
If a system of triples contains two blocks that share a pair, 
we may remove one of them; the number of triples will decrease by one, 
and the number of uncovered pairs will decrease at most by two, 
so the weight of the system will not increase. 
Hence, the minimum weight is attained by a packing system. 
If this system has size $t$, its weight is 
$t+ \frac{1}{2}(\frac{a(a-1)}{2}-3t) = \frac{1}{2}(\frac{a^2-a}{2}-t)$ 
and is minimized when $t=P(a)$. 
When $a$ is even, by \cref{th:packing}, 
$P(a) = \left\lfloor\frac{a}{3}\cdot\frac{a-2}{2}\right\rfloor$, 
so $C_*(a) = \frac{1}{2}\left\lceil\frac{a^2-a}{2} - \frac{a^2-2a}{6}\right\rceil
= \frac{1}{2} \left\lceil\frac{2a^2-a}{6}\right\rceil$. 
When $a \equiv 1,3 \bmod 6$, $P(a) = \frac{a(a-1)}{6}$, 
so $C_*(a) = \frac{a(a-1)}{6}$.
When $a \equiv 5 \bmod 6$, $P(a) = \left\lfloor\frac{a(a-1}{6}\right\rfloor - 1$, 
so $C_*(a) = \frac{1}{2}\left\lceil\frac{a^2-a}{2} - \frac{a^2-a-6}{6}\right\rceil
= \frac{1}{2}\left\lceil\frac{a^2-a+3}{3}\right\rceil$. 
Since $(a^2-a) \equiv 2 \bmod 6$, the last expression is equal to 
$\left\lceil\frac{a^2-a}{6}\right\rceil$. 
\end{proof}

\begin{theorem}\label{th:lower}
$f(a,b) \geq b\,C_*(a)$. 
\end{theorem}

\begin{proof}[\bf{Proof}]
Let $A \cap B = \emptyset$, $|A|=a$, $B=\{y_1,\ldots,y_b\}$, 
and ${\cal Q}$ be an $(A,B)$-system. 
Denote by ${\cal T}_i$ the system of triples $T \subseteq A$ such that 
$T\cup\{y_i\}$ is a block of ${\cal Q}$ ($i=1,\ldots,b$). 
Suppose, a pair $x',x'' \in A$ is not covered by blocks of $T_i$. 
Then there exists $j \ne i$ such that $\{x',x'',y_i,y_j\}$ is a block of ${\cal Q}$. 
Let $p_i$ be the number of pairs in $A$ not covered by ${\cal T}_i$. 
The number of blocks in ${\cal Q}$ that have exactly two elements from $B$ 
is equal to $\frac{1}{2} \sum_{i=1}^b p_i$, 
and the number of blocks that have exactly one element from $B$ is 
$\sum_{i=1}^b |{\cal T}_i|$, 
Therefore, 
$|{\cal Q}| \geq \sum_{i=1}^b (|{\cal T}_i| + \frac{1}{2} p_i)
= \sum_{i=1}^b w({\cal T}_i) \geq b\,C_*(a)$. 
\end{proof}

We will demonstrate in \cref{sec:upper} that 
$f(a,b) = b\,C_*(a)$ when $a$ is odd and $b \geq a-2$, 
or when $a,b$ are even and $b \geq 2a-2$. 
In the case of even $a$ and odd $b$, 
we need to strengthen the lower bound of \cref{th:lower}.

\begin{lemma}\label{th:d/12}
Let ${\cal T}$ be a system of triples on an $a$-element set $A$ 
where $a \equiv 0,2 \bmod 6$. 
Let $d$ be the number of elements $x \in A$ for which there are triples 
$t',t''\in{\cal T}$ such that $|t' \cap t''|=2$, $x \in (t' \cap t'')$. 
Then $w({\cal T}) \geq C_*(a) + \frac{1}{2}\left\lceil\frac{d}{6}\right\rceil$. 
\end{lemma}

\begin{proof}[\bf{Proof}]
Let $l$ be the number of pairs in $A$ 
that are contained in two or more of blocks of ${\cal T}$, 
and let $p$ be the number of pairs, not contained in any blocks. 
Then $l \geq d/2$ and 
$p \geq \binom{a}{2} - 3|{\cal T}| + l$, 
so $|{\cal T}| + \frac{1}{3}\, p \geq \frac{1}{3}\binom{a}{2} + \frac{d}{6}$. 
Since $p \geq \frac{a-d}{2}$, we get 
\begin{align*}
  w({\cal T}) \, = \, |{\cal T}| + \frac{p}{2} 
  \: = \: |{\cal T}| + \frac{p}{3} + \frac{p}{6}
  \:\geq\: \frac{1}{3}\binom{a}{2} + \frac{d}{6} + \frac{a-d}{12}
  \: = \: \frac{2a^2 - a + d}{12} \, .
\end{align*}
As $a \equiv 0,2 \bmod 6$, $\;\frac{2a^2-a}{6}$ is integer. 
Since $2w({\cal T})$ is integer too, we get 
$2w({\cal T}) \geq \frac{2a^2-a}{6} + \left\lceil\frac{d}{6}\right\rceil
  = 2C_*(a) + \left\lceil\frac{d}{6}\right\rceil$.
\end{proof}

\begin{theorem}\label{th:lower+}
For $a \equiv 0,2 \bmod 6$ and odd $b$, 
\begin{align*}
 f(a,b) \geq (b-1)\,C_*(a) + C(a). 
\end{align*}
\end{theorem}

\begin{proof}[\bf{Proof}]
Let $A=\{x_1,\ldots,x_a\}$, $B=\{y_1,\ldots,y_b\}$. 
Let ${\cal Q}$ be an $(A,B)$-system of size $f(a,b)$. 
For $i=1,\ldots,b$, 
let ${\cal T}_i$ be the system of triples $\{x',x'',x'''\} \subseteq A$ such that 
$\{x',x'',x''',y_i\} \in {\cal Q}$. 
For $j,k \in [a]$, $i \neq j$, set $\alpha_i(j,k)=0$ 
if the pair $\{x_j,x_k\}$ is covered by the system ${\cal T}_i$, 
otherwise set $\alpha_i(j,k)=1$. 
Set $m(j,k) := \sum_{i=1}^b \alpha_i(j,k)$. 
Then, similarly to the proof of \cref{th:lower}, 
\begin{align}\label{eq:lower1}
  |Q| \:\geq\: \sum_{i=1}^b |{\cal T}_i| \: + \!\!
  \sum_{\{j,k\}\subseteq [a]} \left\lceil\frac{m(j,k)}{2}\right\rceil .
\end{align}
For $j \in [a]$, set $d_i(j)=1$ if there are triples $t',t''\in{\cal T}_i$ 
such that $|t' \cap t''|=2$ and $x_j \in (t' \cap t'')$. 
By \cref{th:d/12}, 
\begin{align}\label{eq:lower2}
  w({\cal T}_i) \:\geq\: C_*(a) + \frac{1}{12}\sum_{j=1}^a d_i(j) \, .
\end{align}
If $d_i(j)=0$, then $\sum_{k \in [a]\backslash\{j\}} \alpha_i(j,k)$ is odd 
(as $a$ is even). 
Since $b$ is odd, 
if $d_i(j)=0$ for all $i=1,\ldots,b$, 
then $\sum_{i=1}^b \sum_{k \in [a]\backslash\{j\}} \alpha_i(j,k)$ is odd. 
Thus, 
\begin{align*}
  \sum_{i=1}^b \sum_{j=1}^a d_i(j) & \: \geq \:
  \left|\left\{j:\: \sum_{i=1}^b \sum_{k \in [a]\backslash\{j\}} \alpha_i(j,k) \equiv 0 \bmod 2\right\}\right|
  \\ & \: = \: a - 
  \left|\left\{j:\: \sum_{i=1}^b \sum_{k \in [a]\backslash\{j\}} \alpha_i(j,k) \equiv 1 \bmod 2\right\}\right|
  .
\end{align*}
If $\sum_{i=1}^b \sum_{k \in [a]\backslash\{j\}} \alpha_i(j,k) \equiv 1 \bmod 2$, 
then there exists $k \neq j$ such that $\sum_{i=1}^b \alpha_i(j,k) \equiv 1 \bmod 2$. 
As $\sum_{i=1}^b \alpha_i(j,k) = m(j,k)$, we get 
\begin{align}\label{eq:lower3}
  \sum_{i=1}^b \sum_{j=1}^a d_i(j) \:\geq\:
  a - 2 |\{\{j,k\}:\: m(j,k) \equiv 1 \bmod 2\}| \, .
\end{align}
Collecting (\ref{eq:lower1}), (\ref{eq:lower2}), and (\ref{eq:lower3}), we get 
\begin{align*}
  f(a,b) & = |{\cal Q}| \geq \sum_{i=1}^b |{\cal T}_i| \: + \!\!
  \sum_{\{j,k\}\subseteq [a]} \left\lceil\frac{m(j,k)}{2}\right\rceil 
  \\ & =
  \sum_{i=1}^b |{\cal T}_i| \: + \!\!
  \sum_{\{j,k\}\in\binom{[a]}{2}} \frac{m(j,k)}{2}
  \, + \, \frac{1}{2} |\{\{j,k\}:\: m(j,k) \equiv 1 \bmod 2\}|
  \\ & =
  \sum_{i=1}^b w({\cal T}_i)
  \, + \, \frac{1}{2} |\{\{j,k\}:\: m(j,k) \equiv 1 \bmod 2\}|
  \\ & \geq
  b\, C_*(a) + \frac{1}{12} \sum_{i=1}^b \sum_{j=1}^a d_i(j)
  \, + \, \frac{1}{2} \, |\{\{j,k\}:\: m(j,k) \equiv 1 \bmod 2\}|
  \\ & \geq
  b\, C_*(a) + \frac{a}{12} 
  + \left(\frac{1}{2} - \frac{1}{12} \cdot 2 \right)
  \cdot |\{\{j,k\}:\: m(j,k) \equiv 1 \bmod 2\}|
  \\ & \geq
  b\, C_*(a) + \frac{a}{12}
  \, .
\end{align*}
As $b$ is odd, $(b-1)C_*(a)$ is integer. 
As $f(a,b)$ is also integer, we get 
\begin{align*}
  f(a,b) & \geq \left\lceil b\,C_*(a) + \frac{a}{12}\right\rceil
  = (b-1)C_*(a) +  \left\lceil C_*(a) + \frac{a}{12}\right\rceil
  \\ & =
  (b-1)C_*(a) + C(a) \, .
\end{align*}
\end{proof}

\section{Upper bounds on $f(a,b)$}\label{sec:upper}

Following the notation in \cite{Etzion:1994}, we denote 
by $\lambda(n,3)$ the maximum number of disjoint optimal $(n,3,2)$-covering systems 
on the same $n$-element set, 
and by $\mu(n,3)$ the minimum number of optimal $(n,3,2)$-covering systems 
on the same element set such that their union contains all $\binom{n}{3}$ triples. 

\begin{proposition}[{\cite{Etzion:1994,Ji:2006}}]\label{th:_mu}
$\mu(a,3)=a-2$ for $a \neq 6,7$, $\mu(6,3)=5$, $\mu(7,3)=6$. 
\end{proposition}

\begin{proposition}[{\cite{Etzion:1994}}]\label{th:_lambda}
$\lambda(a,3)=a-3$ for $a \equiv 0 \bmod 6$. 
\end{proposition}

\begin{proposition}\label{th:mu}
If $b \geq \mu(a,3)$ then $f(a,b) \leq b\,C(a)$. 
\end{proposition}

\begin{proof}[\bf{Proof}]
Let ${\cal T}_1,\ldots,{\cal T}_b$ be optimal $(a,3,2)$-coverings 
on an $a$-element set $A$ 
such that every triple in $A$ appears as a block in at least one of the coverings. 
Let $B=\{y_1,\ldots,y_b\}$, $B \cap A = \emptyset$. 
Adding $y_i$ to each block of ${\cal T}_i$ for $i=1,\ldots,b$ 
produces a collection of blocks, which is an $(A,B)$-system of size $b\,C(a)$.
\end{proof}

\begin{proposition}\label{th:lambda}
If $b \leq \lambda(a,3)$ then $f(a,b) \leq \binom{a}{3}$. 
\end{proposition}

\begin{proof}[\bf{Proof}]
Let ${\cal T}_1,\ldots,{\cal T}_b$ be disjoint optimal $(a,3,2)$-coverings 
on an $a$-element set $A$. 
There are $m = \binom{a}{3} - b\,C(a)$ triples in $A$ that do not appear as blocks 
in any of the coverings. 
There is a system $\cal{Q}$ of at most $m$ quadruples on $A$ that 
cover these $m$ triples. 
Let $B=\{y_1,\ldots,y_b\}$, $B \cap A = \emptyset$. 
Add $y_i$ to each block of ${\cal T}_i$ for $i=1,\ldots,b$. 
The union of the resulting quadruple systems together with the system $\cal{Q}$ 
is an $(A,B)$-system whose size does not exceed $b\,C(a) + m = \binom{a}{3}$. 
\end{proof}

\begin{proposition}\label{th:+1}
$f(a,b+1) \leq f(a,b) + C(a)$. 
\end{proposition}

\begin{proof}[\bf{Proof}]
Let $|A|=a$, $|B|=b$, and ${\cal Q}$ be an $(A,B)$-system of size $f(a,b)$. 
Let $B'=B \cup \{y\}$ where $y \notin B$. 
Add to ${\cal Q}$ the $C(a)$ quadruples $T \cup {y}$ 
where $T$ belongs to an optimal covering system of triples on $A$. 
It is easy to see that the resulting system of quadruples is an $(A,B')$-system. 
\end{proof}

\begin{proposition}\label{th:+2}
$f(a,b+2) \leq f(a,b) + 2C_*(a)$. 
\end{proposition}

\begin{proof}[\bf{Proof}]
Let $|A|=a$, $|B|=b$, and ${\cal Q}$ be an $(A,B)$-system of size $f(a,b)$. 
Let $B'=B \cup \{y_1,y_2\}$ where $y_1,y_2 \notin B$. 
Let ${\cal T}$ be a system of triples on $A$ of the minimum weight $C_*(a)$. 
Let ${\cal P}$ be a system of the pairs from $A$ that are not covered by ${\cal T}$. 
Add to ${\cal Q}$ quadruples
$T \cup \{y_i\}$ where $T \in {\cal T}$, $i=1,2$, 
and $P \cup \{y_1,y_2\}$ where $P \in {\cal P}$. 
It is easy to see that the resulting system of quadruples is an $(A,B')$-system 
of size 
$f(a,b) + 2\left(|{\cal T}| + \frac{1}{2} |{\cal P}|\right) = f(a,b) + 2C_*(a)$. 
\end{proof}

\begin{proposition}\label{th:2-6mod12}
\begin{align*}
  f(a,a) \:\leq\: a\,C_*(a) \:=\: \frac{2a^3 - a^2}{12} 
  \;\;\;\;{\rm for}\; a \equiv 2,6 \bmod 12
  \, .
\end{align*}
\end{proposition}

\begin{proof}[\bf{Proof}]
Set $n=a/2$. 
Let ${\cal Q}$ be a Steiner $(3,4,n+1)$-system 
(its existence for $n+1 \equiv 2,4 \bmod 6$ is proved in \cite{Hanani:1960}). 
We will use it to construct an $(A,B)$-system  ${\cal Q}'$ of size $(2a^3 - a^2)/12$ 
where $A=\{0,\ldots,a-1\}$, $B=\{a,\ldots,2a-1\}\,$. 

We assume that the elements of ${\cal Q}$ are $0,1,\ldots,n$. 
We distinguish two types of blocks in ${\cal Q}$: 
ones that contain $n$ and ones that do not. 
For each block $\{z_1,z_2,z_3,z_4\}\in{\cal Q}$ that does not contain $n$, 
we add to ${\cal Q}'$ the following $32$ blocks 
with $i_1,i_2,i_3,i_4\in\{0,1\}$ and $i_1+i_2+i_3+i_4 \equiv 0 \bmod 2$: 
\begin{align*}
\{ z_1 + i_1 n,\, z_2 + i_2 n,\, z_3 + i_3 n,\, 2n + z_4 + i_4 n \}, 
\\
\{ z_1 + i_1 n,\, z_2 + i_2 n,\, z_4 + i_4 n,\, 2n + z_3 + i_3 n \}, 
\\
\{ z_1 + i_1 n,\, z_3 + i_3 n,\, z_4 + i_4 n,\, 2n + z_2 + i_2 n \}, 
\\
\{ z_2 + i_2 n,\, z_4 + i_4 n,\, z_4 + i_4 n,\, 2n + z_1 + i_1 n \}. 
\end{align*}
For each block $\{z_1,z_2,z_3,n\}\in{\cal Q}$, 
we add to ${\cal Q}'$ the following $24$ blocks 
with $i_1,i_2,i_3,i_4\in\{0,1\}$, $i_1+i_2+i_3+i_4 \equiv 0 \bmod 2$, 
and $k\in\{1,2,3\}$: 
\begin{align*}
\{ z_1 + i_1 n,\, z_2 + i_2 n,\, z_3 + i_3 n,\, 2n + z_k + i_4 n \}, 
\end{align*}
as well as the $6$ blocks
\begin{align*}
\{ z_i,\, z_j + n,\, 2n + z_k,\, 2n + z_k + n \}, 
\end{align*}
where $(i,j,k)$ is a permutation of $(1,2,3)$. 
Finally, we add to ${\cal Q}'$ the $n$ blocks 
\begin{align*}
\{ i,\, i+n,\, 2n+i,\, 2n+i+n \} 
\end{align*}
with $i=0,1,\ldots,n-1$. 

Using the fact that every triple from $\{0,1,\ldots,n\}$ 
is contained in one of the blocks of ${\cal Q}$, 
it is easy to check that ${\cal Q}'$ is an $(A,B)$-system. 
As $|{\cal Q}| = (n^3-n)/24$ and 
the number of blocks in ${\cal Q}$ that contain $n$ is 
$4|{\cal Q}|/(n+1) = (n^2-n)/6$, we get
\begin{align*}
  |{\cal Q'}| & = \left(\frac{n^3-n}{24} - \frac{n^2-n}{6}\right) \cdot 32
               \: + \: \frac{n^2-n}{6} \cdot (24+6) \: + \: n
  \\ & = \frac{4n^3 - n^2}{3} \: = \: \frac{2a^3 - a^2}{12} \, .
\end{align*}
\end{proof}

\begin{lemma}\label{th:pairs}
Let $A \cap B = \emptyset$, $|A|=a$, $|B|=b$. 
Let ${\cal Q}$ be a system of quadruples on $A \cup B$ such that 
every triple from $A$ is contained in some of the quadruples. 
For $\{x',x''\} \subseteq A$ let $U(x',x'')$ denote the set of elements $y$ from $B$ 
such that $\{x',x'',y\}$ is not contained in any blocks of ${\cal Q}$. 
Then 
\begin{align*}
  f(a,b) \;\leq\; |{\cal Q}| + \!\! \sum_{\{x',x''\} \subseteq\ A} 
    \left\lceil\frac{|U(x',x'')|}{2} \right\rceil \, .
\end{align*}
\end{lemma}

\begin{proof}[\bf{Proof}]
For each pair $\{x',x''\} \subseteq\ A$, add to ${\cal Q}$ 
quadruples $\{x',x'',y_{2i-1},y_{2i}\}$ where $2 \leq 2i \leq k=|U(x',x'')|$ 
and $U(x',x'') = \{y_1,y_2,\ldots,y_k\}$. 
If $k$ is odd, also add $\{x',x'',y_k,z\}$ where $z$ is an element 
of $A \cup B$ distinct from $x',x'',y_k$. 
Then the resulting system of quadruples is an $(A,B)$-system of the required size. 
\end{proof}

\begin{lemma}\label{th:pairs4}
For $a \equiv 4 \bmod 6$, there is a system of $(a^2 - 2a + 4)/6$ triples on $[a]$ 
that leaves uncovered only $(a-2)/2$ disjoint pairs.
\end{lemma}

\begin{proof}[\bf{Proof}]
By \cite[Theorem~3]{Spencer:1968}, an optimal packing of triples on $[a]$ 
is a system of $(a^2 - 2a - 2)/6$ triples 
where the set of uncovered pairs is $K_{1,3}$ plus $(a-4)/2$ independent edges. 
Add to that system a triple which covers two of the edges of this $K_{1,3}$. 
\end{proof}

\begin{proposition}\label{th:0-4mod12}
\begin{align*}
  f(a,a-1) & \:\leq\: \frac{4a^3 - 5a^2}{24} 
  \;\;\;\;{\rm for}\; a \equiv 0 \bmod 12,
  \; a > 12,
  \\
  f(a,a-1) & \:\leq\: \frac{4a^3 - 5a^2 + 16}{24} 
  \;\;\;\;{\rm for}\; a \equiv 4 \bmod 12.
\end{align*}
\end{proposition}

\begin{proof}[\bf{Proof}]
Set $n=a/2$. 
As $n+1 \equiv 1,3 \bmod 6$ and $\:n+1 \neq 7$,
there exist $n-1$ pairwise disjoint Steiner $(2,3,n+1)$-systems 
${\cal T}_0$, ${\cal T}_1$, \ldots, ${\cal T}_{n-2}$ 
on the element set $X=\{0,1,\ldots,n\}$. 
Each triple from $X$ appears as a block in exactly one of these systems. 
We are going to construct an $(A,B)$-system ${\cal Q}$ of the required size 
with $A=\{0,1,\ldots,a-1\}$, $B=\{a+1,\dots,2a-2\}$. 

For each block $\{z_1,z_2,z_3\}\in{\cal T}_j$ that does not contain $j$, 
we add to ${\cal Q}$ the following $8$ blocks
with $i_1,i_2,i_3,i_4\in\{0,1\}$ and $i_1+i_2+i_3+i_4 \equiv 0 \bmod 2$: 
\begin{align*}
\{ z_1 + i_1 n,\, z_2 + i_2 n,\, z_3 + i_3 n,\, 2n + j + i_4 (n-1) \}. 
\end{align*}
For each block $\{z_1,z_2,j\}\in{\cal T}_j$, 
we add to ${\cal Q}$ the following $4$ blocks: 
\begin{align*}
\{ z_1,\, z_1 + n,\, & z_2,\,     2n + j \},
\\
\{ z_1,\, z_1 + n,\, & z_2 + n,\, 2n + j \},
\\
\{ z_2,\, z_2 + n,\, & z_1,\,     2n + j + (n-1) \},
\\
\{ z_2,\, z_2 + n,\, & z_1 + n,\, 2n + j + (n-1) \}.
\end{align*}
So far, we have added 
$(n-1)\left(8(\frac{(n+1)n}{6}-\frac{n}{2}) + 4\cdot\frac{n}{2}\right)
= n(n-1)(8n-4)/6 = a(a-1)(a-2)/6$ 
blocks. 
It is easy to see that every triple from $A$ is contained 
in some of the blocks of ${\cal Q}$. 
Every triple $\{i',i'',j\}$ with $0 \leq i' < i'' \leq 2n-1$, $i'' - i' \neq n$, 
$\:j=2n,\ldots,4n-3$
is also contained in one of the blocks of ${\cal Q}$. 
However, for each pair $\{i,i+n\}$ with $i=0,1,\ldots,n-1$, 
there are $n-1$ values of $j \in \{2n,\ldots,4n-3\}$ such that  
the triple $\{i,i+n,j\}$ is not covered by the blocks of ${\cal Q}$. 
Notice that $n-1$ is odd.

If $a \equiv 0 \bmod 12$, then by \cite[Theorem~3]{Spencer:1968}, 
there exists a system ${\cal T}$ of
$(a^2-2a)/6$ triples in $A$ such that the only pairs uncovered by these triples 
are $\{i,i+n\}$ with $i=0,1,\ldots,n-1$. 
Add to ${\cal Q}$ the $(a^2-2a)/6$ quadruples $t\cup\{4n-2\}$ where $t\in{\cal T}$. 
Then by \cref{th:pairs}, 
\begin{align*}
 f(a,a-1) & \:\leq\: |{\cal Q}| + n \cdot \frac{(n-1)+1}{2}
 \\ & \: = \: \frac{a(a-1)(a-2)}{6} + \frac{a^2-2a}{6} + \frac{a^2}{8} 
 \: = \: \frac{4a^3-5a^2}{24} \, .
\end{align*}

If $a \equiv 4 \bmod 12$, then by \cref{th:pairs4}, 
there exists a system ${\cal T}$ of
$(a^2-2a+4)/6$ triples in $A$ such that the only pairs uncovered by these triples 
are $\{i,i+n\}$ with $i=0,1,\ldots,n-2$. 
Add to ${\cal Q}$ the $(a^2-2a+4)/6$ quadruples $t\cup\{4n-2\}$ where $t\in{\cal T}$. 
Then by \cref{th:pairs}, 
\begin{align*}
 f(a,a-1) & \:\leq\: |{\cal Q}| + n \cdot \frac{n}{2}
 \\ & \: = \: \frac{a(a-1)(a-2)}{6} + \frac{a^2-2a+4}{6} + \frac{a^2}{8}
 \: = \: \frac{4a^3-5a^2+16}{24} \, .
\end{align*}
\end{proof}

\begin{lemma}\label{th:6t+2_or_6t+4}
Let $|A|=a$, $|B|=a-3$, where $A \cap B = \emptyset$, 
$a \equiv 2,4 \bmod 6$, $a \ne 8$. 
Then there exists a system of quadruples on $A \cup B$ of size $(a-3)\,C(a)$ that 
covers all triples with two elements from $A$ and one from $B$ 
as well as all triples from $A$ with the exception of one. 
\end{lemma}

\begin{proof}[\bf{Proof}]
Let $A=\{x_1,\ldots,x_a\}$, $B=\{y_1,\ldots,y_{a-3}\}$, $A' = A \backslash \{x_a\}$. 
By \cref{th:mu}, there is a system of quadruples ${\cal Q}$ of size 
$(a-3)\,C(a-1) = \frac{(a-1)(a-2)(a-3)}{6}$ that covers all triples from $A' \cup B$ 
with at at least two elements from $A'$. 
Add to ${\cal Q}$ blocks 
$\{y_i,x_a,x_{a-1},x_i\}$, $\{y_i,x_a,x_{a-2},x_i\}$ with $i=1,\ldots,a-3$ 
and blocks 
$\{y_i,x_a,x_{i+j},x_{i-j}\}$ with  $i=1,\ldots,a-3$, $j=1,\ldots,\frac{a-4}{2}$, 
where the indices $i \pm j$ are taken modulo $a-3$. 
The resulting system consists of 
$\frac{(a-1)(a-2)(a-3)}{6} + 2(a-3) + \frac{(a-3)(a-4)}{2} = (a-3)\,\frac{a^2-2}{6}$ 
quadruples that 
cover all triples with two elements from $A$ and one from $B$ 
as well as all triples from $A$ with the exception of $\{x_{a-2},x_{a-1},x_a\}$. 
\end{proof}

\begin{proposition}\label{th:6t+2or4}
For $a \equiv 2,4 \bmod 6$, $a \ne 8$, 
and $j \geq 1$, 
\begin{align*}    
  f(a,a-3+2j) & \:\leq\: 
  (a-3)\,C(a) \:+\: 2j\,C_*(a) 
  \\ & \: = \:
  (a-3)\,\frac{a^2+2}{6} \:+\: j\left\lceil\frac{2a^2-a}{6}\right\rceil 
  .
\end{align*}
\end{proposition}

\begin{proof}[\bf{Proof}]
Let $b=a-3+2j$, $A=\{x_1,\ldots,x_a\}$, $B=\{y_1,\ldots,y_b\}$, 
$B'=\{y_1,\ldots,y_{a-3}\}$. 
By \cref{th:6t+2_or_6t+4}, there is a system ${\cal Q}$ 
of $(a-3)\,\frac{a^2+2}{6}$ quadruples on $A \cup B'$ that cover all triples from $A \cup B'$ 
with at at least two elements from $A$, except $\{x_1,x_2,x_3\}$. 
By \cref{th:packing}, there exists a system ${\cal T}$ of 
$m = \left\lfloor\frac{a^2-2a}{6}\right\rfloor$ 
triples on $A$ such that each pair of elements is contained in at most one triple. 
Let ${\cal P}$ be the system of pairs that are not covered by these triples. 
Then $|{\cal P}| = \frac{a(a-1)}{2} - 3m$. 
Add to ${\cal Q}$ 
quadruples $T \cup \{y_i\}$, where $T \in {\cal T}$ and $j=a-4,\ldots,b$, 
as well as quadruples $P \cup \{y_{a-4+2k},\,y_{a-3+2k}\}$, 
where $P \in {\cal P}$ and $k=1,\ldots,j$. 
Then the resulting system is an $(A,B)$-system, where the total number of blocks is 
$(a-3)\,\frac{a^2+2}{6} + 2jm + j(\frac{a(a-1)}{2} - 3m) = (a-3)\,\frac{a^2+2}{6} 
+ j\left\lceil\frac{2a^2-a}{6}\right\rceil$. 
\end{proof}

\begin{proposition}\label{th:S34}
\begin{align*}
 f(a,a) \leq \frac{a^3-a}{6} \, .
\end{align*}
\end{proposition}

\begin{proof}[\bf{Proof}]
Let $A=\ZZ_a$, $B=\{y_j: j\in\ZZ_a\}$. 
Let ${\cal Q}$ be a system of quadruples with two types of blocks: 
$\{i,j,k,y_{i+j+k}\}$ where $\{i,j,k\} \subseteq A$, 
and $\{i,j,y_{2i+j},y_{i+2j}\}$ where $\{i,j\} \subseteq A$. 
It is easy to see that ${\cal Q}$ is an $(A,B)$-system and 
$|{\cal Q}| 
= \binom{a}{3} + \binom{a}{2} = \binom{a+1}{3}$. 
\end{proof}

\begin{proposition}\label{th:6t}
For $a \equiv 0 \bmod 6$, $a \geq 12$, and $j \geq 1$, 
\begin{align*}    
  f(a,a-3+2j) & \:\leq\: 
  (a-3)\,C(a) \:+\: 2j\,C_*(a) 
  \\ & \: = \: 
  (a-3)\,\frac{a^2}{6} \:+\: j\,\frac{2a^2-a}{6} .
\end{align*}
\end{proposition}

\begin{proof}[\bf{Proof}]
Let $a=6t$, $b=6t-3+2j$, $A=\{x_1,\ldots,x_a\}$, $B=\{y_1,\ldots,y_b\}$. 
It is proved in \cite[Section~C]{Etzion:1994} that there are 
systems ${\cal T}_1$,\ldots,${\cal T}_{6t-3}$ 
of triples on $A$ such that that each ${\cal T}_i$ is an optimal $(6t,3,2)$-covering, 
and ${\cal T}_1 \cup \ldots \cup {\cal T}_{6t-3}$ contains all triples in $A$ except 
$\{x_1,x_2,x_3\},\{x_4,x_5,x_6\},\ldots,\{x_{6t-2},x_{6t-1},x_{6t}\}$. 
Set $A' = A \cup \{x_0\}$. 
It is proved in \cite{Vanstone:1993} that for $t \geq 3$, 
there is a Steiner triple system ${\cal S}$ on $A'$ which contains 
$\{x_1,x_2,x_3\},\!\{x_4,x_5,x_6\},\!\ldots,\!\{x_{6t-2},x_{6t-1},x_{6t}\}$ as blocks. 
When $t=2$, it is easy to check that both Steiner triple systems of order $13$ 
have this property as well. 
Denote by ${\cal S}_0$ (by ${\cal S}_1$) the subsystem of ${\cal S}$ 
formed by the blocks that contain (do not contain) $x_0$. 
Consider a system of quadruples ${\cal Q}$ formed by the following blocks:
\begin{itemize}
\item 
Blocks of ${\cal T}_i$ appended by $y_i$ ($i=1,\ldots,6t-3$);
\item 
Blocks of ${\cal S}_0$ appended by $y_i$ ($i=6t-2,\ldots,6t-3+2j$);
\item 
Blocks of ${\cal S}_1$ where $x_0$ is replaced with $y_{6t-4+2k},\,y_{6t-3+2k}$ 
($k \in [j]$). 
\end{itemize}
It is easy to see that ${\cal Q}$ is an $(A,B)$-system, 
and the number of blocks in it is 
$(a-3)C(a) + 2j\,\frac{a^2-2a}{6} + j\frac{a}{2} 
= \frac{(a-3)a^2}{6} + j\,\frac{2a^2-a}{6}$ . 
\end{proof}

\begin{proposition}\label{th:large}
For even $a$, $\;f(a,2a-2) \leq (2a-2)\,C(a)$.
\end{proposition}

\begin{proof}[\bf{Proof}]
Since the case $a=6$ follows from \cref{th:2-6mod12,th:+2}, we may assume $a \neq 6$. 
Consider first the case $a \equiv 0,2 \bmod 6$. 
Then $a+1 \equiv 1,3 \bmod 6$, $\:a+1 \neq 7$. 
There exist (see \cite{Teirlinck:1991}) $(a-1)$ disjoint Steiner $S(2,3,a+1)$-systems 
${\cal T}_1,\ldots,{\cal T}_{a-1}$ on an $(a+1)$-element set $A$. 
Select $x \in A$ and set $A^* = A\backslash\{x\}$, 
$\,{\cal T}_i^* := \{t\in{\cal T}_i:\: x \notin t\}$. 
Notice that every triple from $A^*$ appears as a block in one of the systems 
${\cal T}_1^*,\ldots,{\cal T}_{a-1}^*$. 
For each $i=1,\ldots,a-1$, 
the number of pairs from $A^*$ not covered by blocks of ${\cal T}_i^*$ is 
$\binom{a}{2} - |{\cal T}_i^*|
= \binom{a}{2} - 3(\frac{(a+1)a}{6} - \frac{a}{2}) 
= \frac{a}{2}$. 
Set $B=\{y_1,\ldots,y_{2a-2}\}$ where $\,B \cap A = \emptyset$. 
Consider a system of quadruples ${\cal Q}$ that contains two types of blocks: 
$t\cup\{y_j\}$ where $t\in{\cal T}_i^*$, $\:j\in\{2i-1,2i\}$, 
and $\{x',x'',y_{2i-1},y_{2i}\}$ 
where $\{x',x''\}$ is a pair uncovered by blocks of ${\cal T}_i^*$. 
Every pair $\{x',x''\} \subseteq A^*$ remains uncovered 
in an even number of systems among ${\cal T}_1^*,\ldots,{\cal T}_{a-1}^*$. 
Thus, by \cref{th:pairs}, 
\begin{align*}
  f(a,2a-2) & \:\leq\: (2a-2) \left(\frac{(a+1)a}{6} - \frac{a}{2}\right)
    \: + \: \frac{1}{2} (2a-2)\,\frac{a}{2}
  \\ & = \: (2a-2)\, C(a) \, .
\end{align*}

Now consider the case $a \equiv 4 \bmod 6$. 
By \cite{Etzion:1994,Ji:2006}, there exist $a-1$ disjoint systems of triples 
${\cal T}_1,\ldots,{\cal T}_{a-1}$ 
on an $(a+1)$-element set $A$ such that each ${\cal T}_i$ is an optimal packing, 
and each triple from $A$ appears as a block in at least one of the systems. 
Moreover, there is an element $x \in A$ 
(in fact, there are two elements with this property) 
such that each subsystem
$\,{\cal T}_i^* := \{t\in{\cal T}_i:\: x \notin t\}$ 
is an optimal packing, 
has $\frac{a^2 - 2a - 2}{6}$ blocks 
and leaves uncovered $\frac{a+2}{2}$ pairs from $A^*: = A\backslash\{x\}$. 
The rest of the proof is the same as in the preceding case. 
\end{proof}

When $a \equiv 0,4,8,10 \bmod 12$, 
the upper bounds of \cref{th:6t+2or4,th:6t} can be improved. 
In particular, we are aware of constructions for small $a$ 
that yield the following bounds: 
$f(8,9) \leq 92$, 
$f(10,11) \leq 179$, 
$f(10,13) = 208$, 
$f(12,13) \leq 303$, 
$f(12,15) \leq 348$, 
$f(16,17) \leq 710$, 
$f(16,19) \leq 791$.

\section{Exact values of $f(a,b)$}

\begin{theorem}\label{th:exact}
If $a$ is odd and $b \geq a-2$ $\;(b \geq 6$ if $a=7)$, then 
\begin{align*}
  f(a,b) = b\left\lceil\frac{a^2-a}{6}\right\rceil .
\end{align*} 
\end{theorem}

\begin{proof}[\bf{Proof}]
As $C_*(a) = C(a)$ for odd $a$, the statement of the theorem follows from 
\cref{th:mu,th:_mu,th:C_*,th:lower}. 
\end{proof}

\begin{theorem}\label{th:exact2}
If $a \equiv 2,6 \bmod 12$ and $b \geq a$, then 
\begin{align*}
  f(a,b) & \: = \: b\,\frac{2a^2-a}{12} \;\;\;\;{\rm for\;even}\; b ,
  \\
  f(a,b) & \: = \: (b-1)\,\frac{2a^2-a}{12} 
    + \left\lceil\frac{a^2-a}{6}\right\rceil \;\;\;\;{\rm for\;odd}\; b .
\end{align*} 
\end{theorem}

\begin{proof}[\bf{Proof}]
The statement follows from 
Theorems~\ref{th:lower}, \ref{th:lower+} 
and
Propositions~\ref{th:C_*}, \ref{th:+1}, \ref{th:+2}, \ref{th:2-6mod12}. 
\end{proof}

\begin{theorem}\label{th:exact3}
If $a$ is even  
and $b \geq 2a-2$, then
\begin{align*}
 f(a,b) & \: = \: \frac{b}{2} \left\lceil\frac{2a^2-a}{6}\right\rceil
 \;\;\;\;{\rm if}\; b \equiv 0 \bmod 2,
 \\
 f(a,b) & \: = \: \frac{b-1}{2} \left\lceil\frac{2a^2-a}{6}\right\rceil
   \: + \: \left\lceil\frac{a^2}{6}\right\rceil  
 \;\;\;\;{\rm if}\; b \equiv 1 \bmod 2.
\end{align*} .
\end{theorem}

\begin{proof}[\bf{Proof}]
The statement follows from 
Theorems~\ref{th:lower}, \ref{th:lower+} 
and
Propositions~\ref{th:C_*}, \ref{th:+1}, \ref{th:+2}, \ref{th:large}. 
\end{proof}

\section{$(n,4,3,4)$-lottery systems}\label{sec:lotto}

Let $A,B,C$ be disjoint finite sets. 
Let ${\cal Q'}$ be an $(A,B)$-system, 
${\cal Q''}$ be an $(B,C)$-system, 
and ${\cal Q'''}$ be an $(C,A)$-system. 
It is easy to see that ${\cal Q'} \cup {\cal Q''} \cup {\cal Q'''}$ 
is an $(|A \cup B \cup C|,4,3,4)$-lottery system. 
This simple observation yields an upper bound on $L(n)$: 

\begin{proposition}\label{th:ABC}
$L(a+b+c) \leq f(a,b)+f(b,c)+f(c,a)$. 
\end{proposition}

Thus, the upper bounds on $f(a,b)$ from \cref{sec:upper} 
can be used to obtain bounds on $L(n)$. 
The trick is to find the best partition $n=a+b+c$ for each $n$ 
(the order of parts also matters). 
For example, $18t+15 = (6t+7)+(6t+5)+(6t+3)$ gives
better results than $18t+15 = (6t+5)+(6t+5)+(6t+5)$ 
or $18t+15 = (6t+7)+(6t+3)+(6t+5)$. 

\begin{theorem}\label{th:lotto}
\begin{align*}
L(n) &\:\leq\: (4n^3 -12n^2 + 48n + 176)/216 \;\;\;\;{\rm if}\; n \equiv 1\bmod 18,
 \;n>19,\\
L(n) &\:\leq\: (4n^3 -12n^2          )/216 \;\;\;\;{\rm if}\; n \equiv 3,9\bmod 18,\\
L(n) &\:\leq\: (4n^3 -12n^2 +        16)/216 \;\;\;\;{\rm if}\; n \equiv 5 \bmod 18,\\
L(n) &\:\leq\: (4n^3 -12n^2 +        80)/216 \;\;\;\;{\rm if}\; n \equiv 7 \bmod 18,\\
L(n) &\:\leq\: (4n^3 -12n^2 + 48n -  80)/216\;\;\;\;{\rm if}\; n \equiv 11\bmod 18,\\
L(n) &\:\leq\: (4n^3 -12n^2 + 96n -  16)/216\;\;\;\;{\rm if}\; n \equiv 13\bmod 18,\\
L(n) &\:\leq\: (4n^3 -12n^2 + 48n + 144)/216\;\;\;\;{\rm if}\; n \equiv 15\bmod 18,\\
L(n) &\:\leq\: (4n^3 -12n^2 + 96n + 112)/216\;\;\;\;{\rm if}\; n \equiv 17\bmod 18,
 \;n>17,
\end{align*}
\begin{align*}
L(n) &\:\leq\: (4n^3- 9n^2 + 48n      )/216 \;\;\;\;{\rm if}\; n \equiv 0\bmod 36,
       \: n>36,\\
L(n) &\:\leq\: (4n^3- 9n^2 - 12n +  28)/216 \;\;\;\;{\rm if}\; n \equiv 2\bmod 36,
       \: n>38,\\
L(n) &\:\leq\: (4n^3- 9n^2 - 12n + 152)/216 \;\;\;\;{\rm if}\; n \equiv 4\bmod 36,
       \: n>40,\\
L(n) &\:\leq\: (4n^3-10n^2 + 12n +  72)/216 \;\;\;\;{\rm if}\; n \equiv 6\bmod 36,\\
L(n) &\:\leq\: (4n^3-10n^2 +  4n +  72)/216 \;\;\;\;{\rm if}\; n \equiv 8\bmod 36,\\
L(n) &\:\leq\: (4n^3- 9n^2       + 140)/216\;\;\;\;{\rm if}\; n \equiv 10\bmod 36,\\
L(n) &\:\leq\: (4n^3- 9n^2 + 36n + 216)/216\;\;\;\;{\rm if}\; n \equiv 12\bmod 36,\\
L(n) &\:\leq\: (4n^3- 9n^2 + 84n + 196)/216\;\;\;\;{\rm if}\; n \equiv 14\bmod 36,\\
L(n) &\:\leq\: (4n^3- 9n^2 + 36n + 248)/216\;\;\;\;{\rm if}\; n \equiv 16\bmod 36,\\
L(n) &\:\leq\: (4n^3-10n^2 + 60n      )/216\;\;\;\;{\rm if}\; n \equiv 18\bmod 36,\\
L(n) &\:\leq\: (4n^3-10n^2 +  4n      )/216\;\;\;\;{\rm if}\; n \equiv 20\bmod 36,\\
L(n) &\:\leq\: (4n^3-10n^2 +  8n +  88)/216\;\;\;\;{\rm if}\; n \equiv 22\bmod 36,\\
L(n) &\:\leq\: (4n^3- 8n^2       - 144)/216\;\;\;\;{\rm if}\; n \equiv 24\bmod 36,\\
L(n) &\:\leq\: (4n^3- 8n^2 - 16n + 104)/216\;\;\;\;{\rm if}\; n \equiv 26\bmod 36,\\
L(n) &\:\leq\: (4n^3- 8n^2 + 16n - 120)/216\;\;\;\;{\rm if}\; n \equiv 28\bmod 36,\\
L(n) &\:\leq\: (4n^3- 8n^2 + 36n +  72)/216\;\;\;\;{\rm if}\; n \equiv 30\bmod 36,\\
L(n) &\:\leq\: (4n^3- 8n^2 + 68n + 224)/216\;\;\;\;{\rm if}\; n \equiv 32\bmod 36,\\
L(n) &\:\leq\: (4n^3- 8n^2 +  4n + 504)/216\;\;\;\;{\rm if}\; n \equiv 34\bmod 36.
\end{align*}
\end{theorem}

\begin{proof}[\bf{Proof}]
\Cref{th:exact} yields
\begin{align*}
f(6t-1,6t+1) & = 36t^3 -  12t^2 +   3t +  1, \\ 
f(6t+1,6t-1) & = 36t^3          -    t \;\;\;\; (t \geq 2),\\ 
f(6t+1,6t  ) & = 36t^3 +   6t^2,             \\ 
f(6t+1,6t+1) & = 36t^3 +  12t^2 +    t,      \\ 
f(6t+1,6t+3) & = 36t^3 +  24t^2 +   3t,      \\ 
f(6t+3,6t+1) & = 36t^3 +  36t^2 +  11t +  1, \\ 
f(6t+3,6t+2) & = 36t^3 +  42t^2 +  16t +  2, \\ 
f(6t+3,6t+3) & = 36t^3 +  48t^2 +  21t +  3, \\ 
f(6t+3,6t+4) & = 36t^3 +  54t^2 +  26t +  4, \\ 
f(6t+3,6t+5) & = 36t^3 +  60t^2 +  31t +  5, \\ 
f(6t+3,6t+6) & = 36t^3 +  66t^2 +  36t +  6, \\ 
f(6t+3,6t+7) & = 36t^3 +  72t^2 +  41t +  7, \\ 
f(6t+5,6t+3) & = 36t^3 +  72t^2 +  51t + 12, \\ 
f(6t+5,6t+5) & = 36t^3 +  84t^2 +  69t + 20, \\ 
f(6t+5,6t+6) & = 36t^3 +  90t^2 +  78t + 24, \\ 
f(6t+5,6t+7) & = 36t^3 +  96t^2 +  87t + 28, \\ 
f(6t+7,6t+5) & = 36t^3 + 108t^2 + 107t + 35, \\ 
f(6t+7,6t+6) & = 36t^3 + 114t^2 + 120t + 42. 
\end{align*}
\Cref{th:0-4mod12,th:+2} yield
\begin{align*}
f(12s  ,12s+1) & \leq 288s^3 +  18s^2 -   2s \;\;\;\; (s \geq 2), \\ 
f(12s  ,12s+3) & \leq 288s^3 +  66s^2 -   4s \;\;\;\; (s \geq 2), \\
f(12s+4,12s+3) & \leq 288s^3 + 258s^2 +  76s +  8, \\
f(12s+4,12s+5) & \leq 288s^3 + 306s^2 + 106s + 13, \\
f(12s+4,12s+7) & \leq 288s^3 + 354s^2 + 136s + 18.
\end{align*}
\Cref{th:exact2} yields
\begin{align*}
f(12s+2,12s+3) & = 288s^3 + 156s^2 +  28s +  2,\\ 
f(12s+6,12s+7) & = 288s^3 + 444s^2 + 228s + 39,\\
f(12s+6,12s+9) & = 288s^3 + 492s^2 + 274s + 50. 
\end{align*}
\Cref{th:6t,th:+2} yield 
\begin{align*}
f(12s+ 8,12s+7) & \leq 288s^3 + 552s^2 + 354s +  75, \\ 
f(12s+10,12s+9) & \leq 288s^3 + 696s^2 + 562s + 151.
\end{align*}
\Cref{th:S34,th:+2} yield 
\begin{align*}
f(12s+ 8,12s+9) & \leq 288s^3 + 600s^2 + 414s +  95.
\end{align*}
\Cref{th:lambda,th:_lambda} yield 
\begin{align*}
f(12s+12,12s+9) & \leq 288s^3 + 792s^2 + 724s + 220.
\end{align*}
To complete the proof for odd $n$, we use \cref{th:ABC} with the following partitions: 
\begin{align*}
  L(18t+1) & \leq f(6t+1,6t+1)+f(6t+1,6t-1)+f(6t-1,6t+1) \\
           & = 108t^3          +   3t +  1 \\
           & = \frac{1}{54} ((18t+1)^3 - 3(18t+1)^2 + 12(18t+1) + 44) \, ,
\end{align*}
\begin{align*}
  L(18t+3) & \leq 3f(6t+1,6t+1)
           \: = \: 108t^3 +  36t^2 +   3t \\
           & = \frac{1}{54} ((18t+3)^3 - 3(18t+3)^2) \, ,
\end{align*}
\begin{align*}
  L(18t+5) & \leq f(6t+3,6t+1)+f(6t+1,6t+1)+f(6t+1,6t+3) \\
           & = 108t^3 +  72t^2 +  15t +  1 \\
           & = \frac{1}{54} ((18t+5)^3 - 3(18t+5)^2 + 4) \, ,
\end{align*}
\begin{align*}
  L(18t+7) & \leq f(6t+3,6t+3)+f(6t+3,6t+1)+f(6t+1,6t+3) \\
           & = 108t^3 + 108t^2 +  35t +  4 \\
           & = \frac{1}{54} ((18t+7)^3 - 3(18t+7)^2 + 20) \, ,
\end{align*}
\begin{align*}
  L(18t+9) & \leq 3f(6t+3,6t+3) \\
           & = 108t^3 + 144t^2 +  63t +  9 \\
           & = \frac{1}{54} ((18t+9)^3 - 3(18t+9)^2) \, ,
\end{align*}
\begin{align*}
  L(18t+11) & \leq f(6t+5,6t+3)+f(6t+3,6t+3)+f(6t+3,6t+5) \\
            & = 108t^3 + 180t^2 + 103t + 20 \\
            & = \frac{1}{54} ((18t+11)^3 - 3(18t+11)^2 + 12(18t+11) - 20) \, ,
\end{align*}
\begin{align*}
  L(18t+13) & \leq f(6t+5,6t+5)+f(6t+5,6t+3)+f(6t+3,6t+5) \\
            & = 108t^3 + 216t^2 + 151t + 37 \\
            & = \frac{1}{54} ((18t+13)^3 - 3(18t+13)^2 + 24(18t+13) - 4) \, ,
\end{align*}
\begin{align*}
  L(18t+15) & \leq f(6t+7,6t+5)+f(6t+5,6t+3)+f(6t+3,6t+7) \\
            & = 108t^3 + 252t^2 + 199t + 54 \\
            & = \frac{1}{54} ((18t+15)^3 - 3(18t+15)^2 + 12(18t+15) + 36) \, ,
\end{align*}
\begin{align*}
  L(18t+17) & \leq f(6t+7,6t+5)+f(6t+5,6t+5)+f(6t+5,6t+7) \\
            & = 108t^3 + 288t^2 + 263t + 83 \\
            & = \frac{1}{54} ((18t+17)^3 - 3(18t+17)^2 + 24(18t+17) + 28) \, .
\end{align*}
To complete the proof for even values of $n$, we use \cref{th:ABC} with the following partitions: 
\begin{align*}
  L(36s)   & \leq f(12s+1,12s-1)+f(12s-1,12s)+f(12s,12s+1)
  \\ & \leq (288s^3 - 2s) + (288s^3 - 72s^2 + 12s) + (288s^3 + 18s^2 - 2s)
  \\ & = 864s^3 - 54s^2 + 8s
  \: = \: \frac{1}{216} (4(36s)^3 - 9(36s)^2 + 48(36s)) \, ,
\end{align*}
\begin{align*}
  L(36s+2)   & \leq f(12s+1,12s+1)+f(12s+1,12s)+f(12s,12s+1)
  \\ & \leq (288s^3 + 48s^2 + 2s) + (288s^3 + 24s^2) + (288s^3 + 18s^2 - 2s)
  \\ & = 864s^3 + 90s^2
  \\ & = \frac{1}{216} (4(36s+2)^3 - 9(36s+2)^2 - 12(36s+2) + 28) \, ,
\end{align*}
\begin{align*}
  L(36s+4)   & \leq f(12s+3,12s+1)+f(12s+1,12s)+f(12s,12s+3)
  \\ & \leq (288s^3 + 144s^2 + 22s + 1) + (288s^3 + 24s^2) 
  \\ & \;\;\;\;\;\;\;\;\;\; + (288s^3 + 66s^2 - 4s)
  \\ & = 864s^3 + 234s^2 + 18s + 1
  \\ & = \frac{1}{216} (4(36s+4)^3 - 9(36s+4)^2 - 12(36s+4) + 152) \, ,
\end{align*}
\begin{align*}
  L(36s+6)   & \leq f(12s+3,12s+1)+f(12s+1,12s+2)
  \\ & \;\;\;\;\;\;\;\;\;\; +f(12s+2,12s+3)
  \\ & = (288s^3 + 144s^2 + 22s + 1) + (288s^3 + 72s^2 + 4s) 
  \\ & \;\;\;\;\;\;\;\;\;\; + (288s^3 + 156s^2 + 28s + 2)
  \\ & = 864s^3 + 372s^2 + 54s + 3
  \\ & = \frac{1}{216} (4(36s+6)^3 - 10(36s+6)^2 + 12(36s+6) + 72) \, ,
\end{align*}
\begin{align*}
  L(36s+8)   & \leq f(12s+3,12s+3)+f(12s+3,12s+2)
  \\ & \;\;\;\;\;\;\;\;\;\; +f(12s+2,12s+3)
  \\ & = (288s^3 + 192s^2 + 42s + 3) + (288s^3 + 168s^2 + 32s + 2) 
  \\ & \;\;\;\;\;\;\;\;\;\; + (288s^3 + 156s^2 + 28s + 2)
  \\ & = 864s^3 + 516s^2 + 102s + 7
  \\ & = \frac{1}{216} (4(36s+8)^3 - 10(36s+8)^2 + 4(36s+8) + 72) \, ,
\end{align*}
\begin{align*}
  L(36s+10)   & \leq f(12s+3,12s+3)+f(12s+3,12s+4)
  \\ & \;\;\;\;\;\;\;\;\;\; +f(12s+4,12s+3)
  \\ & \leq (288s^3 + 192s^2 + 42s + 3) + (288s^3 + 216s^2 + 52s + 4) 
  \\ & \;\;\;\;\;\;\;\;\;\; + (288s^3 + 258s^2 + 76s + 8)
  \\ & = 864s^3 + 666s^2 + 170s + 15
  \\ & = \frac{1}{216} (4(36s+10)^3 - 9(36s+10)^2 + 140) \, ,
\end{align*}
\begin{align*}
  L(36s+12)   & \leq f(12s+3,12s+5)+f(12s+5,12s+4)
  \\ & \;\;\;\;\;\;\;\;\;\; +f(12s+4,12s+3)
  \\ & \leq (288s^3 + 240s^2 + 62s + 5) + (288s^3 + 312s^2 + 120s + 16) 
  \\ & \;\;\;\;\;\;\;\;\;\; + (288s^3 + 258s^2 + 76s + 8)
  \\ & = 864s^3 + 810s^2 + 258s + 29
  \\ & = \frac{1}{216} (4(36s+12)^3 - 9(36s+12)^2 + 36(36s+12) + 216) \, ,
\end{align*}
\begin{align*}
  L(36s+14)   & \leq f(12s+5,12s+5)+f(12s+5,12s+4)
  \\ & \;\;\;\;\;\;\;\;\;\; +f(12s+4,12s+5)
  \\ & \leq (288s^3 + 336s^2 + 138s + 20) + (288s^3 + 312s^2 + 120s + 16) 
  \\ & \;\;\;\;\;\;\;\;\;\; + (288s^3 + 306s^2 + 106s + 13)
  \\ & = 864s^3 + 954s^2 + 364s + 49
  \\ & = \frac{1}{216} (4(36s+14)^3 - 9(36s+14)^2 + 84(36s+14) + 196) \, ,
\end{align*}
\begin{align*}
  L(36s+16)   & \leq f(12s+7,12s+5)+f(12s+5,12s+4)
  \\ & \;\;\;\;\;\;\;\;\;\; +f(12s+4,12s+7)
  \\ & \leq (288s^3 + 432s^2 + 214s + 35) + (288s^3 + 312s^2 + 120s + 16) 
  \\ & \;\;\;\;\;\;\;\;\;\; + (288s^3 + 354s^2 + 136s + 18)
  \\ & = 864s^3 + 1098s^2 + 470s + 69
  \\ & = \frac{1}{216} (4(36s+16)^3 - 9(36s+16)^2 + 36(36s+16) + 248) \, ,
\end{align*}
\begin{align*}
  L(36s+18)   & \leq f(12s+7,12s+5)+f(12s+5,12s+6)
  \\ & \;\;\;\;\;\;\;\;\;\; +f(12s+6,12s+7)
  \\ & = (288s^3 + 432s^2 + 214s + 35) + (288s^3 + 360s^2 + 156s + 24) 
  \\ & \;\;\;\;\;\;\;\;\;\; + (288s^3 + 444s^2 + 228s + 39)
  \\ & = 864s^3 + 1236s^2 + 598s + 98
  \\ & = \frac{1}{216} (4(36s+18)^3 - 10(36s+18)^2 + 60(36s+18)) \, ,
\end{align*}
\begin{align*}
  L(36s+20)   & \leq f(12s+7,12s+7)+f(12s+7,12s+6)
  \\ & \;\;\;\;\;\;\;\;\;\; +f(12s+6,12s+7)
  \\ & = (288s^3 + 480s^2 + 266s + 49) + (288s^3 + 456s^2 + 240s + 42) 
  \\ & \;\;\;\;\;\;\;\;\;\; + (288s^3 + 444s^2 + 228s + 39)
  \\ & = 864s^3 + 1380s^2 + 734s + 130
  \\ & = \frac{1}{216} (4(36s+20)^3 - 10(36s+20)^2 + 4(36s+20)) \, ,
\end{align*}
\begin{align*}
  L(36s+22)   & \leq f(12s+9,12s+7)+f(12s+7,12s+6)
  \\ & \;\;\;\;\;\;\;\;\;\; +f(12s+6,12s+9)
  \\ & = (288s^3 + 576s^2 + 382s + 84) + (288s^3 + 456s^2 + 240s + 42) 
  \\ & \;\;\;\;\;\;\;\;\;\; + (288s^3 + 492s^2 + 274s + 50)
  \\ & = 864s^3 + 1524s^2 + 896s + 176
  \\ & = \frac{1}{216} (4(36s+22)^3 - 10(36s+22)^2 + 8(36s+22) + 88) \, ,
\end{align*}
\begin{align*}
  L(36s+24)   & \leq f(12s+7,12s+9)+f(12s+9,12s+8)
  \\ & \;\;\;\;\;\;\;\;\;\; +f(12s+8,12s+7)
  \\ & \leq (288s^3 + 528s^2 + 318s + 63) + (288s^3 + 600s^2 + 416s + 96) 
  \\ & \;\;\;\;\;\;\;\;\;\; + (288s^3 + 552s^2 + 354s + 75)
  \\ & = 864s^3 + 1680s^2 + 1088s + 234
  \\ & = \frac{1}{216} (4(36s+24)^3 - 8(36s+24)^2 - 144) \, ,
\end{align*}
\begin{align*}
  L(36s+26)   & \leq f(12s+9,12s+9)+f(12s+9,12s+8)
  \\ & \;\;\;\;\;\;\;\;\;\; +f(12s+8,12s+9)
  \\ & \leq (288s^3 + 624s^2 + 450s + 108) + (288s^3 + 600s^2 + 416s + 96) 
  \\ & \;\;\;\;\;\;\;\;\;\; + (288s^3 + 600s^2 + 414s + 95)
  \\ & = 864s^3 + 1824s^2 + 1280s + 299
  \\ & = \frac{1}{216} (4(36s+26)^3 - 8(36s+26)^2 - 16(s+26) + 104) \, ,
\end{align*}
\begin{align*}
  L(36s+28)   & \leq f(12s+9,12s+9)+f(12s+9,12s+10)
  \\ & \;\;\;\;\;\;\;\;\;\; +f(12s+10,12s+9)
  \\ & \leq (288s^3 + 624s^2 + 450s + 108)
  \\ & \;\;\;\;\;\;\;\;\;\; + (288s^3 + 648s^2 + 484s + 120) 
  \\ & \;\;\;\;\;\;\;\;\;\; + (288s^3 + 696s^2 + 562s + 151)
  \\ & = 864s^3 + 1968s^2 + 1496s + 379
  \\ & = \frac{1}{216} (4(36s+28)^3 - 8(36s+28)^2 + 16(s+28) - 120) \, ,
\end{align*}
\begin{align*}
  L(36s+30)   & \leq f(12s+9,12s+9)+f(12s+9,12s+12)
  \\ & \;\;\;\;\;\;\;\;\;\; +f(12s+12,12s+9)
  \\ & \leq (288s^3 + 624s^2 + 450s + 108)
  \\ & \;\;\;\;\;\;\;\;\;\; + (288s^3 + 696s^2 + 552s + 144) 
  \\ & \;\;\;\;\;\;\;\;\;\; + (288s^3 + 792s^2 + 724s + 220)
  \\ & = 864s^3 + 2112s^2 + 1726s + 472
  \\ & = \frac{1}{216} (4(36s+30)^3 - 8(36s+30)^2 + 36(s+30) + 72) \, ,
\end{align*}
\begin{align*}
  L(36s+32)   & \leq f(12s+9,12s+11)+f(12s+11,12s+12)
  \\ & \;\;\;\;\;\;\;\;\;\; +f(12s+12,12s+9)
  \\ & \leq (288s^3 + 672s^2 + 518s + 132)
  \\ & \;\;\;\;\;\;\;\;\;\; + (288s^3 + 792s^2 + 732s + 228) 
  \\ & \;\;\;\;\;\;\;\;\;\; + (288s^3 + 792s^2 + 724s + 220)
  \\ & = 864s^3 + 2256s^2 + 1974s + 580
  \\ & = \frac{1}{216} (4(36s+32)^3 - 8(36s+32)^2 + 68(s+32) + 224) \, ,
\end{align*}
\begin{align*}
  L(36s+34)   & \leq f(12s+9,12s+13)+f(12s+13,12s+12)
  \\ & \;\;\;\;\;\;\;\;\;\; +f(12s+12,12s+9)
  \\ & \leq (288s^3 + 720s^2 + 586s + 156)
  \\ & \;\;\;\;\;\;\;\;\;\; + (288s^3 + 888s^2 + 912s + 312) 
  \\ & \;\;\;\;\;\;\;\;\;\; + (288s^3 + 792s^2 + 724s + 220)
  \\ & = 864s^3 + 2400s^2 + 2222s + 688
  \\ & = \frac{1}{216} (4(36s+34)^3 - 8(36s+34)^2 + 4(s+34) + 504) \, .
\end{align*}

\end{proof}

\section{Concluding remarks}\label{sec:concluding}

We call a system of $r$-element subsets of $A \cup B$ (where $A \cap B = \emptyset$) 
an $(A,B)$-system if each triple $t \subseteq A \cup B$ with $|t \cap A| \geq 2$ 
is contained in one of these subsets. 
Let $f_r(a,b)$ denote the smallest size of such a system when $|A|=a$, $|B|=b$. 
Similarly to \cref{th:ABC}, 
\begin{align}\label{eq:general}
  L(a+b+c,r,3,4) \:\leq\: f_r(a,b) + f_r(b,c) + f_r(c,a) \, .
\end{align}
Let $X \subseteq A \cup B$, $\:|X|=r$. 
The number of triples $\{x',x'',y\}$ it contains with $x',x'' \in A$, $y \in B$ 
is $\binom{t}{2}(r-t)$ where $t=|X \cap A|$. 
This number is maximized when 
$t=2m$ if $r=3m$, 
$t=2m+1$ if $r=3m+1$, 
$t=2m+1$ or $t=2m+2$ if $r=3m+2$. 
Let $t=t(r)$ be an optimal value. 
Then
\begin{align}\label{eq:lower_general}
  f_r(a,b) \:\geq\: \frac{b \binom{a}{2}}{(r-t(r)) \binom{t(r)}{2}} \, .
\end{align}
One can show that there exists a construction that yields 
\begin{align*}
  f_r(a,a) \:\leq\: \frac{a \binom{a}{2}}{(r-t(r)) \binom{t(r)}{2}} \, (1+o(1))
  \;\;\;{\rm when}\; a \to \infty \, .
\end{align*} 

In some instances, the equality in \cref{eq:lower_general} can be established.

\begin{example}
Let ${\cal S}$ be a Steiner $S(3,5,a+1)$-system 
with the element set $[a+1]$ 
(such system are known for $a=4,16,25,64,100$). 
Set $A=[a]$, $B=\{y_1,\ldots,y_a\}$. 
For each $i=1,\ldots,a$,
we take each block of ${\cal S}$ that contains $i$ and then 
(a) remove $i$, 
(b) if the block contains $a+1$, replace it with $i$,
(c) add $y_i$.
Notice that for a fixed $i$, such blocks cover all pairs in $A$.
It is also easy to see that every triple from $A$ is covered by one of the blocks of the constructed system. 
Therefore, 
$f_5(a,a) \leq \frac{a^2 (a-1)}{12}$. 
The matching upper bound follows from \cref{eq:lower_general}, so
$f_5(a,a) = \frac{a^2 (a-1)}{12}$. 
By \cref{eq:general}, 
$L(3a,5,3,4) \leq \frac{a^2 (a-1)}{4}$, 
which is the best known upper bound for the abovementioned values of $a$. 
\end{example}

\begin{example}
A Steiner $S(t,k,a)$-system is called \emph{$p$-resolvable} 
if its block-set can be partitioned into Steiner $S(p,k,a)$ systems ($p \leq t$). 
An $S(3,4,a)$ can be $2$-resolvable only if $a \equiv 4 \bmod 12$, 
and its resolution consists of $(a-2)/2$ Steiner $S(2,4,a)$-systems whose blocks together cover all triples from the underlying $a$-element set. 
It is known (see \cite{Semakov:1971,Baker:1976,Teirlinck:1994}) 
that $2$-resolvable $S(3,4,a)$ exist for 
$a = 4^m$ and $a=2q^m + 2$ with $q=7,31,127$. 
(The case $q=127$ follows from 
Proposition~3.2 of \cite{Teirlinck:1994} with $K=\{127\}$.) 
Hence, for any such $a$ and any even $b \geq a-2$, 
there exist $b/2$ Steiner $S(2,4,a)$-systems 
whose blocks cover all triples from the underlying element set $A=[a]$. 
Let $B=\{y_1,\ldots,y_b\}$. 
Pad blocks from $S(2,4,a)$-system number $j$ with a pair ${y_{2j-1}, y_{2j}}$ 
to get $\frac{b}{2} \cdot \frac{a(a-1)}{12}$ blocks of size $6$ 
that cover all triples from $A$ 
as well as all triples with two elements in $A$ and one in $B$. 
Thus, $f_6(a,b) \leq \frac{ba(a-1)}{24}$. 
The matching upper bound follows from \cref{eq:lower_general}, so
$f_6(a,b) = \frac{ba(a-1)}{24}$. 
By \cref{eq:general}, 
$L(3a,6,3,4) \leq \frac{a^2(a-1)}{8}$, 
which is presently the best upper bound. 
\end{example}

\vspace{4mm}
{\bf Acknowledgments.}
No funding was received for conducting this study.
The author is grateful to Iliya Bluskov for helpful discussions.

\vspace{4mm}
{\bf Conflict of Interests Statement.} 
The author has no relevant financial or non-financial interests to disclose.

\vspace{4mm}
{\bf Data Accessibility Statement.}
The article describes entirely theoretical research, so no data sets have been used. 
Constructions of some $(n,4,3,4)$-lottery systems built on the basis of 
\cref{th:lotto} can be found on the website \cite{Italian:tables}.

\end{document}